\documentclass[12pt]{amsart} 

\usepackage[all]{xy}

\usepackage[english]{babel}
\usepackage[latin1]{inputenc}

\usepackage{amsmath}
\usepackage{amsthm}
\usepackage{amssymb}
\usepackage{tikz}
\usepackage{thmtools}
\usepackage{imakeidx} 
\usepackage{appendix}
\usepackage{tikz-cd}
\usepackage{verbatim}
\usepackage{enumitem}
\usepackage{multicol}
\usepackage{textgreek}
\usepackage{csquotes}

\usepackage[backref=page]{hyperref}
\usepackage{cleveref}

\hypersetup{colorlinks=true,linkcolor=blue,anchorcolor=blue,citecolor=blue}

\usepackage{adjustbox}

\usepackage{mathtools}

\newtheorem{thm}{Theorem}[section]
\newtheorem{prop}[thm]{Proposition}
\newtheorem{lemma}[thm]{Lemma}
\newtheorem{cor}[thm]{Corollary}

\newtheorem{question}[thm]{Question}

\newtheorem{defin}[thm]{Definition}

\newtheorem{claim}[thm]{Claim}

\theoremstyle{remark}
\newtheorem{rmk}[thm]{Remark}
\numberwithin{equation}{section}

\newcommand{\F}{\mathbb F}

\newcommand{\Z}{\mathbb Z}

\newcommand{\Br}{\operatorname{Br}}

\renewcommand{\phi}{\varphi}
\newcommand{\Mat}{\operatorname{M}}
\newcommand{\on}[1]{\operatorname{#1}}

\author{Zinovy Reichstein}
\address[Reichstein]{Department of Mathematics\\
	University of British Columbia\\
	Vancouver, BC V6T 1Z2\\Canada}
\email{reichst@math.ubc.ca}
\thanks{Zinovy Reichstein was partially supported by Natural Sciences and Engineering Research Council of Canada Discovery grant  RGPIN-2023-03353.}

\author{Federico Scavia}
\address[Scavia]{CNRS\\
	Institut Galil\'ee\\
	Universit\'e Sorbonne Paris Nord\\
	99 avenue Jean-Baptiste Cl\'ement, 93430\\ 
	Villetaneuse, France}
\email{scavia@math.univ-paris13.fr}



\title{Brauer classes not split by genus one curves}

\begin{document}

	\begin{abstract}
		We show that there exist Brauer classes over a field $F$ which are not split by any genus one curve over $F$. This answers a question of Clark and Saltman.
	\end{abstract}
	
	\maketitle
	
	\section{Introduction}
	
	Let $F$ be a field. By a \emph{genus one curve} over $F$, we shall mean a smooth projective geometrically connected curve of geometric genus one over $F$. Given a Brauer class $\alpha\in \on{Br}(F)$ and a smooth integral $F$-variety $V$, we shall say that $\alpha$ is split by $V$ if $\alpha$ pulls back to zero in $\on{Br}(V)$, or equivalently, by a theorem of Grothendieck, if $\alpha$ pulls back to zero in $\on{Br}(F(V))$. The following question, attributed to Clark \cite{clark2008open} and Saltman \cite{ruozzi2011rage}, was inspired by the earlier work of Artin on Severi--Brauer varieties \cite{artin-bs-varieties}.
	
	\begin{question}[Clark, Saltman]
		\label{question-clark-saltman}
		Let $F$ be a field, and let $\alpha\in \Br(F)$ be a Brauer class. Does there exist a genus one curve $C$ over $F$ such that $\alpha$ is split by $C$?
	\end{question}
	
	An equivalent reformulation of \Cref{question-clark-saltman}
	is as follows: Does every Severi--Brauer variety $Y$ over $F$ admit an $F$-morphism $C \to Y$, for some genus one curve $C$ over $F$?
	
	\medskip
	There is now a significant body of literature devoted to \Cref{question-clark-saltman} and its variants. In particular, \Cref{question-clark-saltman} has a positive answer in the following cases:
	\begin{itemize}
		\item enen $\alpha$ has index $\leqslant 3$, by work of Swets \cite{swets1995global};
		\item when $\alpha$ has index $\leqslant 5$, due to de Jong and Ho \cite{dejong2012genus};
		\item when $\alpha$ has index $6$, due to Auel (unpublished);
		\item when $\alpha$ has index $7$ and $F$ is a global field, due to Antieau and Auel \cite{antieau-auel2021};
		\item when $\alpha$ is cyclic of any exponent $n$ and there exists an elliptic curve $E$ over $F$ such that $E[n]\cong \Z/n\Z\times\mu_n$, again due to Antieau and
		Auel \cite[Theorem C]{antieau-auel2021}.
	\end{itemize}
	
	Recall that the Jacobian of a genus one curve is an elliptic curve, and every genus one curve is a torsor under its Jacobian.
	One may thus ask a related (but weaker) question: Can every Brauer class $\alpha \in \Br(F)$ be split by a torsor under some abelian variety over $F$, not necessarily one-dimensional? This question has been answered in the affirmative by Ho and Lieblich~\cite{ho-lieblich2021}. 
	
	For other recent work related to \Cref{question-clark-saltman}, see Saltman \cite{saltman2021genus}, Mackall--Rekuski \cite{mackall2024elliptic}, Mackall \cite{mackall2023chow, mackall2026}, and Huybrechts--Mattei~\cite[Theorem 1.2]{huybrechts-mattei2025}.
	
	In this paper we show that \Cref{question-clark-saltman} has a negative answer. Let $p$ be a prime, and suppose that $F$ contains a primitive $p$-th root of unity $\zeta$. Then, for all $u,v\in F^\times$, we let $(u,v)_p\in \on{Br}(F)[p]$ be the corresponding symbol, that is, the Brauer class of the degree-$p$ cyclic central simple algebra over $F$ generated by two variables $x,y$ modulo the relations $x^p=u$, $y^p=v$ and $xy=\zeta yx$. 
	
	\begin{thm}\label{mainthm}
		Let $p$ be a prime, let $k$ be a field of characteristic different from $p$, and suppose that $k$ contains a primitive $p$th root of unity $\zeta$. 
       Let
		$F_r \coloneqq k(\!(t_1)\!)(\!(t_2)\!) \ldots (\!(t_{2r})\!)$ be the iterated Laurent series field 
        in the variables $t_1, \ldots, t_{2r}$ over $k$, 
        and consider the Brauer class $\alpha_r=(t_1,t_2)_p+ \cdots+ (t_{2r-1},t_{2r})_p \in \Br(F_r)$. If
		\[
		r \geqslant 
		\begin{cases}
			5g^2 +2g & (p>2,\, g\geqslant 2), \\
			9g^2 + 2g-1 & (p = 2,\, g\geqslant 2), \\
			6 & (p>2,\, g=1), \\
			7 & (p=2,\, g=1),
		\end{cases}
		\]
		then $\alpha_r$ cannot be split by any torsor under any $g$-dimensional abelian variety over $F_r$.
	\end{thm}
	
	The case $g=1$ of \Cref{mainthm} gives a negative answer to \Cref{question-clark-saltman}. Since $\alpha_r$ descends to the purely transcendental extension $k(t_1,\dots,t_{2r})$ of $k$, \Cref{mainthm} also gives examples of Brauer classes over $k(t_1,\dots,t_{2r})$ which are not split by genus one curves.
	
    We will now briefly describe the idea of our proof of \Cref{mainthm}. We may assume without loss of generality that $k$ is algebraically closed. For the sake of simplicity we will limit ourselves to the case where $g=1$ and $p$ is odd.  

	Suppose that $\alpha_r$ is split by a genus one curve $C$ over $F_r$. 
    Our starting point is the following elementary observation: If $C$ acquires an $L$-point for some field extension $L/F_r$, then $L$ splits $\alpha_r$. In particular, we will consider finite extensions $L/F_r$ of the form
\begin{equation}\label{e.tower}
\vcenter{
  \xymatrix{
    L \ar@{-}[d]^{\text{\scriptsize Galois}} \\
    F' \ar@{-}[d]^{\text{\scriptsize prime-to-$p$}} \\
    F,
  }
}
\end{equation}
where $F = F_r$ and the degree $[F':F]$ is not divisible by $p$. On the one hand, we show that for any field $F$ of characteristic different from $p$ and any 
genus one curve $C/F$, there exists a tower of the form~\eqref{e.tower} such that $C(L) \neq \emptyset$ and $\on{Gal}(L/F')$ does not contain a subgroup isomorphic to $(\Z/p\Z)^6$; see \Cref{prime-to-p-torsor}, \Cref{split-torsor-abelian-variety} and \Cref{prop-p-rank}. On the other hand, if a field extension $L/F_r$ of this form splits $\alpha_r$, then the Galois group $\on{Gal}(L/F')$ contains a subgroup isomorphic to $(\Z/p\Z)^r$; see \Cref{lem.amitsur}.  This yields a contradiction as soon as $r\geqslant 6$.
	
	To reconcile the counterexamples of \Cref{mainthm} with the known cases listed above, where a positive answer to~\Cref{question-clark-saltman} has been established, we remark that in all of those cases, $\alpha$ is assumed to be the Brauer class of an algebra that is cyclic or bicyclic (or becomes cyclic or bicyclic after a suitable prime-to-$p$ extension). This is precisely the situation where our argument fails most dramatically.

	\section{Splitting torsors under finite \'etale group schemes}\label{section-3}
	The purpose of this section is to prove \Cref{split-completely-torsor}, which describes the Galois groups of some splitting fields for torsors under finite \'etale group schemes. 
	
	We begin with some preliminaries on Galois cohomology. Throughout this section $F$ will denote a field, $F_s$ will denote a separable closure of $F$, and $\Gamma_F\coloneqq \on{Gal}(F_s/F)$ will denote the absolute Galois group of $F$.
    
    For every finite (abstract) group $G_0$, we let $H^1(F,G_0)$ be the set of isomorphism classes of $G_0$-torsors over $F$. We have a natural bijection
	\begin{equation}\label{h1}
		\begin{tikzcd}
			H^1(F,G_0) \arrow[r,"\sim"]& \on{Hom}_{\on{cts}}(\Gamma_F,G_0)\big/{\sim},
		\end{tikzcd}
	\end{equation}
	where on the right $G_0$ is endowed with the discrete topology and, for any two continuous homomorphisms $f_1,f_2\colon \Gamma_F\to G_0$, we say that $f_1\sim f_2$ if and only if there exists $g\in G_0$ such that $f_2(\sigma)= gf_1(\sigma)g^{-1}$ for every $\sigma\in \Gamma_F$. The bijection (\ref{h1}) sends a $G_0$-torsor $P$ to the conjugacy class of any cocycle obtained by fixing a trivialization of $P$ over $F_s$ (such a cocycle is continuous because $P$ becomes trivial over a finite extension of $F$, and it is a homomorphism because $\Gamma_F$ acts trivially on $G_0$). The trivial $G_0$-torsor corresponds to the conjugacy class of the trivial homomorphism $\Gamma_F \to G_0$; this class is a singleton, as conjugation by any $g \in G_0$ fixes the trivial homomorphism.
	
	The bijection (\ref{h1}) is functorial in $F$, in the following sense: for every finite extension $K/F$, we have a commutative square
	\[
	\begin{tikzcd}
		H^1(F,G_0) \arrow[r,"\sim"] \arrow[d] & \on{Hom}_{\on{cts}}(\Gamma_F,G_0)\big/{\sim} \arrow[d] \\
		H^1(K,G_0) \arrow[r,"\sim"] & \on{Hom}_{\on{cts}}(\Gamma_K,G_0)\big/{\sim},
	\end{tikzcd}
	\]
	where the left vertical map is given by base change from $F$ to $K$ and the right vertical map by restriction from $\Gamma_F$ to the open subgroup $\Gamma_{K}$.
	
	Let $\on{Aut}(G_0)$ be the automorphism group of $G_0$. By Galois descent, $H^1(F,\on{Aut}(G_0))$ classifies forms of $G_0$, that is, $F$-groups $G$ such that $G_{F_s}$ is isomorphic to the constant $F_s$-group associated to $G_0$. More precisely, we have bijections
	\begin{equation}\label{forms}
	\begin{tikzcd}[column sep=2em]
		\{\text{Forms of $G_0/F$}\}\big/\mathrm{Isom.} \arrow[r,"\sim"] & H^1(F,\on{Aut}(G_0)) \arrow[r,"\sim"] & \on{Hom}_{\on{cts}}(\Gamma_F,\on{Aut}(G_0))\big/{\sim},
	\end{tikzcd}
\end{equation}
	where the left map sends a form $G$ of $G_0$ to the $\on{Aut}(G_0)$-torsor of isomorphisms between $G_0$ and $G$, and where the right map is (\ref{h1}) for $\on{Aut}(G_0)$.
	The constant $F$-group associated to $G_0$ corresponds under (\ref{forms}) to the trivial $\on{Aut}(G_0)$-torsor and to the trivial homomorphism $\Gamma_F\to \on{Aut}(G_0)$.
	
	For a finite \'etale $F$-scheme $Y$, the \emph{decomposition field} of $Y$ is the smallest (with respect to inclusion) field extension $F\subset L\subset F_s$ such that $Y_L$ is a disjoint union of copies of $\on{Spec}(L)$. Given any separable polynomial $f(x)\in F[x]$ such that  $Y\cong\on{Spec}(F[x]/(f(x)))$, the decomposition field of $Y$ is equal to the splitting field of $f(x)$, that is, the extension of $F$ generated by the roots of $f(x)$ in $F_s$. In particular, the decomposition field of $Y$ is a finite Galois extension of $F$.
	
	\begin{lemma}\label{galois-of-decomposition}
		Let $F$ be a field, let $G_0$ be a finite (abstract) group, let $P$ be a $G_0$-torsor over $F$, let $\varphi\colon \Gamma_F\to G_0$ be a continuous homomorphism corresponding to $P$ via (\ref{h1}), and let $F\subset L\subset F_s$ be the decomposition field of $P$. Then $L$ is the fixed field of $\on{Ker}(\varphi)$, and hence $\on{Gal}(L/F)\cong\on{Im}(\varphi)$. In particular, $\on{Gal}(L/F)$ is isomorphic to a subgroup of $G_0$.
	\end{lemma}
	
	\begin{proof}
		For every field extension $F\subset K\subset F_s$, the following are equivalent:
		\begin{enumerate}
			\item[(i)] $P_K$ is a disjoint union of $|G_0|$ copies of $\on{Spec}(K)$,
			\item[(ii)] $P_K$ is a trivial $G_0$-torsor,
			\item[(iii)] the restriction of $\varphi$ to $\Gamma_K\subset \Gamma_F$ is trivial, that is, $\Gamma_K\subset\on{Ker}(\varphi)$.
		\end{enumerate}
		Indeed, the equivalence between (i) and (ii) follows from the fact that $G_0$ is discrete, and the equivalence between (ii) and (iii) follows from (\ref{h1}) applied over $K$. By definition, $L$ is the smallest field extension of $F$ such that (i) holds. It follows that $L$ is also the smallest extension of $F$ for which (iii) holds, that is, $\Gamma_L=\on{Ker}(\varphi)$ as subgroups of $\Gamma_F$, as desired. 
	\end{proof}
	
	\begin{prop}\label{split-completely-torsor}
		Let $F$ be a field, let $G$ be a finite \'etale $F$-group, let $G_0\coloneqq G(F_s)$, and let $\varphi\colon \Gamma_F\to \on{Aut}(G_0)$ be a continuous homomorphism corresponding to $G$ via (\ref{forms}). Let $P$ be a $G$-torsor over $F$, and let $F\subset L\subset F_s$ be the decomposition field of $P$. Then $P(L)\neq\emptyset$ and the Galois group $\on{Gal}(L/F)$ is an extension of $\on{Im}(\varphi)\leqslant \on{Aut}(G_0)$ by a subgroup of $G_0$.
	\end{prop}
	
	\begin{proof}
		Since $P_L$ is a non-empty disjoint union of copies of $\on{Spec}(L)$, we have $P(L)\neq \emptyset$. Let $F\subset K\subset F_s$ be the decomposition field of $G$.
		
		\begin{claim}\label{claim-1}
			We have $K\subset L$. In particular, $L$ is also the decomposition field of the finite \'etale $K$-scheme $P_K$.
		\end{claim}
		
		\begin{proof}
			Since $P$ is a $G$-torsor, the $G$-action on $P$ induces an isomorphism of $L$-schemes $G_L\times_LP_L\xrightarrow{\sim} P_L\times_LP_L$. The domain of this isomorphism is a disjoint union of copies of $G_L$, while the codomain is a disjoint union of copies of $\on{Spec}(L)$. We deduce that $G_L$ is a disjoint union of copies of $\on{Spec}(L)$. Thus $K\subset L$, as claimed.
		\end{proof}
		By \Cref{claim-1} and Galois theory, we obtain a short exact sequence
		\begin{equation}\label{sequence-galois-theory}
			\begin{tikzcd}
				1 \arrow[r] & \on{Gal}(L/K) \arrow[r] & \on{Gal}(L/F) \arrow[r] & \on{Gal}(K/F) \arrow[r] & 1.    
			\end{tikzcd}
		\end{equation}
		It remains to identify the first and third groups in (\ref{sequence-galois-theory}). By \Cref{galois-of-decomposition}, applied to the discrete group $\on{Aut}(G_0)$ and the $\on{Aut}(G_0)$-torsor of isomorphisms between $G_0$ and $G$, we have $\on{Gal}(K/F) \cong \on{Im}(\varphi)$. 
        By \Cref{claim-1}, $L$ is the decomposition field of the finite \'etale $K$-scheme $P_K$. Since $G_K$ is discrete, it is the constant $K$-group associated to $G(K)=G_0$. By \Cref{galois-of-decomposition}, applied to $G_0$ over $K$, we deduce that $\on{Gal}(L/K)$ is isomorphic to a subgroup of $G_0$.
	\end{proof}
	
\section{Splitting torsors under abelian varieties}
	Let $F$ be a field, and let $A$ be an abelian variety over $F$. A restriction-corestriction argument shows that the abelian group $H^1(F,A)$ is torsion. For every $n\geqslant 1$, the short exact sequence
	\begin{equation}\label{multiplication-by-n}
		\begin{tikzcd}
			0 \arrow[r] & A[n] \arrow[r] & A \arrow[r,"\times n"] & A \arrow[r] & 0  
		\end{tikzcd}
	\end{equation}
	shows that the natural map $H^1(F,A[n])\to H^1(F,A)[n]$ is surjective.
	
	\begin{prop}\label{split-torsor-abelian-variety}
		Let $F$ be a field, let $A$ be an abelian variety of dimension $g\geqslant 1$, let $T$ be an $A$-torsor over $F$, and let $n\geqslant 1$ be the period of the class of $T$ in $H^1(F,A)$. Assume that $F$ contains a primitive $n$th root of unity. Then there exists a Galois field extension $L/F$ such that $T(L)\neq\emptyset$ and such that $\on{Gal}(L/F)$ is an extension of a subgroup of $\mathrm{SL}_{2g}(\Z/n\Z)$ by a subgroup of $(\Z/n\Z)^{2g}$. 
	\end{prop}

	\begin{proof}
		We apply \Cref{split-completely-torsor} to the finite $F$-group $G\coloneqq A[n]$, which is \'etale because $n$ is invertible in $F$, and to an $A[n]$-torsor $P$ inducing $T$ via (\ref{multiplication-by-n}). Thus $G_0\coloneqq G(F_s)$ is isomorphic to $(\Z/n\Z)^{2g}$, so that $\on{Aut}(G_0)\cong \mathrm{GL}_{2g}(\Z/n\Z)$. Let $F\subset L\subset F_s$ be the decomposition field of $P$, and let $\varphi\colon \Gamma_F\to \mathrm{GL}_{2g}(\Z/n\Z)$ be a homomorphism corresponding to $G$ via (\ref{forms}). Then $T(L)\neq\emptyset$ and, by \Cref{split-completely-torsor}, the Galois group $\on{Gal}(L/F)$ is an extension of $\on{Im}(\varphi)$ by a subgroup of $(\Z/n\Z)^{2g}$. It remains to show that $\on{Im}(\varphi)$ is contained in $\mathrm{SL}_{2g}(\Z/n\Z)$. This follows from the fact that $F$ contains a primitive $n$th root of unity. The proof is standard; we include it here for lack of reference.
		
		For every $n$-torsion $\Gamma_F$-module $M$, we let $M^\vee\coloneqq \on{Hom}(M,\Z/n\Z)$ be the dual $\Gamma_F$-module of $M$. By \cite[Theorem 15.1(i)]{milne1986abelian}, we have an isomorphism of $\Gamma_F$-modules 
		\[G(F_s)=A[n](F_s)\cong H^1_{\textrm{\'et}}(A_{F_s},\Z/n\Z)^\vee.\]
		By the computation of the \'etale cohomology of abelian varieties over separably closed fields \cite[Theorem 15.1(ii)]{milne1986abelian} and Poincar\'e duality \cite[Theorem 11.1]{milne1980etale}, we have $\Gamma_F$-equivariant isomorphisms \[\Lambda^{2g}H^1_{\textrm{\'et}}(A_{F_s},\Z/n\Z)\cong H^{2g}_{\textrm{\'et}}(A_{F_s},\Z/n\Z)\cong \mu_n(F_s)^{\otimes (-g)}\cong \Z/n\Z.\]
		Thus $\Lambda^{2g}(G(F_s))\cong \Z/n\Z$ as $\Gamma_F$-modules, that is, the $\Gamma_F$-action on $\Lambda^{2g}(G(F_s))$ is trivial. Equivalently, the image of $\varphi$ is contained in $\mathrm{SL}_{2g}(\Z/n\Z)$, as desired.
	\end{proof}

    \begin{rmk}
        With the notation of \Cref{split-torsor-abelian-variety}, if $A$ is principally polarized, then the Weil pairing endows $A[n](F_s)$ with a Galois-invariant non-degenerate symplectic form. Consequently, $\on{Gal}(L/F)$ is an extension of a subgroup of $\on{Sp}_{2g}(\Z/n\Z)$ by a subgroup of $(\Z/n\Z)^{2g}$. However, when $g\geqslant 2$, not every abelian variety is principally polarized, even over an algebraically closed field, and therefore we cannot, in general, replace $\mathrm{SL}_{2g}(\Z/n\Z)$ by $\on{Sp}_{2g}(\Z/n\Z)$ in \Cref{split-torsor-abelian-variety}.
    \end{rmk}
	
	\section{The \texorpdfstring{$p$}{p}-rank of \texorpdfstring{$\mathrm{SL}_{2g}(\Z/p^e\Z)$}{SL2g(Z/peZ)}}\label{section-4}
	
	Let $p$ be a prime, let $e$ be a positive integer, and let $G$ be a smooth affine group over the ring of $p$-adic integers $\Z_p$. For each $1\leqslant j\leqslant e-1$, the reduction map \[\pi_j\colon G(\Z/p^e\Z) \to G(\Z/p^j\Z)\] is a surjective group homomorphism. We let $H_j\leqslant G(\Z/p^e\Z)$ be the kernel of $\pi_j$.
	
	\begin{lemma}\label{claim-order-p}
		Let $p$ be a prime, let $N\geqslant 1$ be an integer, let $G$ be a smooth closed subgroup of $\mathrm{GL}_N$  of relative dimension $d$ over the ring of $p$-adic integers $\Z_p$.
		\begin{enumerate}
			\item If $p$ is odd, every element of order $p$ in $H_1$ lies in $H_{e-1}$. 
			\item If $p=2$ and $e\geqslant 2$, every element of order $2$ in $H_2$ lies in $H_{e-1}$.
		\end{enumerate}
	\end{lemma}
	
	\begin{proof}
		Using the closed subgroup embedding of $G$ in $\mathrm{GL}_N$, we see that it suffices to prove \Cref{claim-order-p} when $G=\mathrm{GL}_N$. In this case, for every $N \geqslant 1$ we have a short exact sequence
		\[
		\begin{tikzcd}
			0\arrow[r] &  \Mat_N(\Z/p\Z) \arrow[r,"\iota"] & \mathrm{GL}_N(\Z/p^e\Z)\arrow[r,"\pi_{e-1}"] &\mathrm{GL}_N(\Z/p^{e-1}\Z)\arrow[r] &  1.
		\end{tikzcd}
		\]
		Here $\Mat_N(\Z/p\Z)$ denotes the abelian group of $N\times N$ matrices with coefficients in $\Z/p\Z$, and the map $\iota$ sends $M\in \Mat_N(\Z/p\Z)$ to $I+p^{e-1}\tilde{M}$, where $\tilde{M}\in \Mat_N(\Z/p^e\Z)$ is a lift of $M$ and $I$ denotes the $N\times N$ identity matrix.
		
		It suffices to prove, for $1\leqslant j \leqslant e-2$ if $p$ is odd and for $2\leqslant j \leqslant e- 2$ if $p=2$, that every order-$p$ element in $H_j$ lies in $H_{j+1}$. By definition, $B \in \mathrm{GL}_N(\Z/p^e\Z)$ is in $H_j$ if and only if it is of the form 
		\[ B = I + p^j A \]
		for some $A \in \Mat_N(\Z/p^e\Z)$. Our goal is to show that if $B$ has order $p$ in $\mathrm{GL}_N(\Z/p^e \Z)$, then $A$ is divisible by $p$, that is, $A$ lies in $p\cdot\Mat_N(\Z/ p^e \Z)$. If we can prove this, then $B$ lies in
		$H_{j+1}$, and we are done.
		
		Our assumption that $B$ is of order $p$ in $\mathrm{GL}_N(\Z/p^e\Z)$ translates to
		\begin{equation} \label{e.p-power} I = B^p = (I + p^j A)^p = I +  \binom{p}{1} p^j A + \binom{p}{2} p^{2j} A^2 + \binom{p}{3} p^{3j} A^3 + \cdots + \binom{p}{p} p^{pj} A^p. \end{equation}
		
		(1) Suppose that $p$ is odd and $1 \leqslant j \leqslant e-2$. Then every term on the right hand side of~\eqref{e.p-power}, after the first two, is divisible by $p^{j + 2}$. Reducing both sides modulo $p^{j + 2}$, we obtain $I = I + p^{j+1} A$ in $\mathrm{GL}_N(\Z/p^{j+2} \Z)$. Thus $A$ is divisible by $p$, as desired. 
		
		(2) Now suppose that $p=2$ and $2 \leqslant j \leqslant e-2$. In this case~\eqref{e.p-power} takes the form $I=B^2=I+2^{j+1}A+2^{2j}A^2$. Reducing both sides modulo $2^{j + 2}$, we obtain $I = I + 2^{j+1} A$ in $\mathrm{GL}_N(\Z/2^{j+2} \Z)$. Thus $A$ is divisible by $2$, as desired. 
	\end{proof}

\begin{defin}	For a prime $p$ and a finite group $\Gamma$, the $p$-rank of $\Gamma$, denoted by $\on{rank}_p(\Gamma)$, is the largest integer $r\geqslant 0$ such that $\Gamma$ has a subgroup isomorphic to $(\Z/p\Z)^r$.
\end{defin}

	\begin{prop}\label{prop-p-rank} Let $p$ be a prime, let $N\geqslant 1$ be an integer, let $G$ be a smooth closed subgroup of $\mathrm{GL}_N$  of relative dimension $d$ over the ring of $p$-adic integers $\Z_p$. We have
		\[
		\on{rank}_p(G(\Z/p^e\Z)) \leqslant 
		\begin{cases}
			d+\on{rank}_p(G(\Z/p\Z)) & (p>2,\, e\geqslant 1), \\
			d+\on{rank}_2(G(\Z/4\Z)) & (p = 2,\, e\geqslant 2).
		\end{cases}
		\]
	\end{prop}
	
	\begin{proof} 		
		The cases $(p>2,\, e=1)$ and $(p=2,\, e=2)$ are immediate. Thus we may assume
that $e\geqslant 2$ if $p$ is odd, and $e\geqslant 3$ if $p=2$. Let $\Lambda \cong (\Z/ p \Z)^r$ be a subgroup of $G(\Z/ p^e \Z)$, and let $\mathfrak{g}$ be the Lie algebra of the mod $p$ reduction of $G$. Since $G$ is smooth over $\Z_p$, by \cite[Lemme 4.1.1]{begueri1980dualite} (see also \cite[Proposition III.4.3]{milne2006arithmetic}) we have a group isomorphism $H_{e-1}\cong \mathfrak{g}(\F_p)$. Thus $H_{e-1}\cong (\Z/p\Z)^d$, so that $\on{rank}_p(H_{e-1}\cap\Lambda)\leqslant d$. 
        
        Suppose first that $p$ is odd. Since $\pi_1(\Lambda)$ is a subgroup of $G(\Z/p\Z)$, we have the inequality $\on{rank}_p(\pi_1(\Lambda))\leqslant \on{rank}_p(G(\Z/p\Z))$.  As $\Lambda$ is $p$-torsion, by \Cref{claim-order-p}(1) we have a short exact sequence
		\[
		\begin{tikzcd}
			0\arrow[r] & H_{e-1}\cap\Lambda \arrow[r] & \Lambda \arrow[r,"\pi_1"] & \pi_1(\Lambda) \arrow[r] &  0.
		\end{tikzcd}
		\]
		We conclude that \[r=\on{rank}_p(H_{e-1}\cap\Lambda)+\on{rank}_p(\pi_1(\Lambda))\leqslant  d+\on{rank}_p(G(\Z/p\Z)),\]
		as desired. This completes the proof for $p>2$ and $e\geqslant 1$. The proof for the case $p=2$ and $e\geqslant 2$ is analogous, with $\pi_1$ replaced by $\pi_2$ and \Cref{claim-order-p}(1) by \Cref{claim-order-p}(2) in the above argument. 
	\end{proof}

	\begin{cor}\label{cor-p-rank}
		Let $p$ be a prime, and let $g\geqslant 1$ and $e\geqslant 0$ be integers.
		\begin{enumerate}
			\item 
        We have
			\[
	\on{rank}_p(\mathrm{SL}_{2g}(\Z/p^e\Z)) \leqslant 
			\begin{cases}
				5g^2-1 & (p>2), \\
				9g^2-2 & (p = 2).
			\end{cases}
	\]
			\item 
 We have
			\[
		\on{rank}_p(\mathrm{SL}_2(\Z/p^e\Z)) \leqslant 
			\begin{cases}
				3 & (p>2), \\
				4 & (p = 2),
			\end{cases}
			\]
            with equality if $(p>2,e\geqslant 2)$ or $(p=2,e\geqslant 3)$.
		\end{enumerate}
	\end{cor}
	
	\begin{proof}
When $e=0$, there is nothing to prove. From now on, we assume $e\geqslant 1$.
    
		(1)  By~\cite[Table 3.3.1 on p.~108]{gorenstein1998classification}, $\on{rank}_p(\mathrm{SL}_{2g}(\Z/p\Z))=g^2$ for any prime $p$ and any $g\geqslant 1$. Applying \Cref{prop-p-rank} to 
		$G = \mathrm{SL}_{2g}$ and remembering that $d = \dim(G) = 4g^2-1$, we obtain the desired inequalities when $p>2$ and also when $p=2$ and $e=1$.
		
		Now suppose that $p=2$ and $e\geqslant 2$. From the short exact sequence
		\[\begin{tikzcd}
			0\arrow[r] & \Mat_{2g}(\Z/2\Z)_0 \arrow[r] & \mathrm{SL}_{2g}(\Z/4\Z) \arrow[r] & \mathrm{SL}_{2g}(\Z/2\Z) \arrow[r] & 1    
		\end{tikzcd}\]
		where $\Mat_{2g}(\Z/2\Z)_0\subset \Mat_{2g}(\Z/2\Z)$ denotes the subgroup of trace-zero matrices, we deduce 
		\[\on{rank}_2(\mathrm{SL}_{2g}(\Z/4\Z))\leqslant \on{rank}_2(\Mat_{2g}(\Z/2\Z)_0)+\on{rank}_2(\mathrm{SL}_{2g}(\Z/2\Z))=4g^2-1+g^2=5g^2-1.\] 
		The desired inequality now follows from \Cref{prop-p-rank} for any $e \geqslant 2$.
		
		\smallskip
		(2) Since the order of $\mathrm{SL}_2(\Z/p\Z)$ is $p(p^2-1)$, 
        any $p$-Sylow subgroup of $\mathrm{SL}_2(\Z/p\Z)$ is  conjugate to the cyclic subgroup $U$ of order $p$ generated by 
        $\begin{pmatrix} 1 & 1 \\
        0 & 1 \end{pmatrix}$. In particular, we may assume that $e\geqslant 2$.
  
  First assume that $p$ is odd. Since $\mathrm{SL}_2(\Z/ p^e \Z)$ contains $H_{e-1} \cong (\Z/ p \Z)^3$, we have $3 \leqslant \on{rank}_p (\mathrm{SL}_2(\Z/p^e\Z)) \leqslant 4$, where the last inequality follows from part (1). Our goal is thus to show that $\on{rank}_p(\mathrm{SL}_2(\Z/p^e\Z))$ cannot be $4$.
We argue by contradiction: suppose that $\mathrm{SL}_2(\Z/ p^e \Z)$ contains a subgroup $\Lambda \cong (\Z/p\Z)^4$. Consider the exact sequence 
\[\begin{tikzcd}
    0 \arrow[r] &  \Lambda \cap H_1 \arrow[r] & \Lambda \arrow[r] & \pi_1(\Lambda) \arrow[r] & 0.
\end{tikzcd}
\]
By \Cref{claim-order-p}(1),  $H_1 \cap \Lambda = H_{e-1}\cap \Lambda$ is a subgroup of $H_{e-1} \cong (\Z/p \Z)^3$. Thus $\Lambda \cong (\Z/ p\Z)^4$ is only possible if $\Lambda$ contains $H_{e-1}$ and $\pi_1(\Lambda)$ is conjugate to $U$.  After replacing $\Lambda$ by a conjugate subgroup in $\mathrm{SL}_2(\Z/p^e\Z)$, we may assume that $\Lambda$ contains both $H_{e-1}$ and a matrix $B$ of the form 
\[     B = \begin{pmatrix} 1 & 1 \\
        0 & 1 \end{pmatrix} + pA \]
for some $A \in \Mat_2(\Z/p^e\Z)$. 
 However, it is easy to see that $B$ does not commute with
every element of $H_{e-1}$; in particular, it does not commute with 
$\begin{pmatrix} 1 & 0 \\ p^{e-1} & 1 \end{pmatrix}$,
a contradiction. 

		Now suppose that $p=2$ and $e = 2$. Any $2$-Sylow subgroup of $\mathrm{SL}_2(\Z/4\Z)$ is an extension of $\Z/2\Z$ by $H_1\cong (\Z/2\Z)^3$. Therefore $3\leqslant \on{rank}_2(\mathrm{SL}_2(\Z/4\Z))\leqslant 4$. Since $\mathrm{SL}_2(\Z/4\Z)$ contains elements of order $4$, its $2$-Sylow subgroups are not elementary abelian, and hence  $\on{rank}_2(\mathrm{SL}_2(\Z/4\Z))=3$. 

Finally suppose that $p = 2$ and $e\geqslant 3$. In this case we will prove the following stronger assertion: The elements of order dividing $2$ in $\mathrm{SL}_2(\Z/2^e\Z)$ form a subgroup isomorphic to $(\Z/2\Z)^4$. For this, let $M \in \mathrm{SL}_2(\Z/2^e\Z)$ be such that $M^2=I$. By the Cayley--Hamilton theorem,
\[ M^2 - \on{tr}(M) M + \det(M) I = 0 \quad \text{in $\Mat_2(\Z/ 2^e \Z)$}. \]
Multiplying both sides by $M$, and remembering that $M^2 = I$ and $\det(M) = 1$, we obtain $2M = \on{tr}(M) I$ in $\Mat_2(\Z/ 2^e \Z)$. Thus $\on{tr}(M) = 2u$ in $\Z/ 2^e \Z$ for some odd integer $u$, and  
\[ \pi_{e-1}(M) =  u I \; \; \text{in} \; \, \Mat_2(\Z/2^{e-1} \Z) . \]
In other words, 
\[			M = \begin{pmatrix}
				u+2^{e-1}n_{11} & 2^{e-1}n_{12} \\
				2^{e-1}n_{21} & u+2^{e-1}n_{22}
			\end{pmatrix},      \]
for some integers $n_{11}, n_{12}, n_{21}, n_{22} \in\{0,1\}$. Moreover, since $\on{tr}(M) = 2u$ in $\Z/2^e \Z$, we conclude that $n_{11} = n_{22}$.
The condition that $M$ lies in $\mathrm{SL}_2(\Z/ 2^e \Z)$ translates 
to $u^2 \equiv 1 \pmod{2^e}$. Thus $u \equiv \pm 1\pmod{2^e}$ or $u\equiv \pm 1 + 2^{e-1}
\pmod{2^e}$. After possibly changing the value of $n_{11} = n_{22}$, we may assume that $u \equiv \pm 1 \pmod{2^e}$. In summary, every order $2$ element in $\mathrm{SL}_2(\Z/2^e\Z)$ is 
of the form
\begin{equation}\label{matrix-order-2}
			M = \begin{pmatrix}
				u+2^{e-1}n_{11} & 2^{e-1}n_{12} \\
				2^{e-1}n_{21} & u+2^{e-1}n_{11}
			\end{pmatrix},
		\end{equation}
where $n_{11}, n_{12}, n_{21} \in \{0,1\}$ and $u \equiv \pm 1 \pmod{2^e}$. 
Conversely, a direct computation shows that the $16$ matrices $M \in \mathrm{SL}_2(\Z/2^e\Z)$ of this form
have order $1$ or $2$ and commute with each other. Moreover, if we denote the set of all such matrices by $S$, then $S = \pi_{e-1}^{-1}(\{\pm I\})$ is a subgroup of $\mathrm{SL}_2(\Z/2^e\Z)$ of order $16$.
We conclude that $S \cong (\Z/2\Z)^4$. This completes the proof of \Cref{cor-p-rank}.
\end{proof}

\section{Splitting fields of \texorpdfstring{$\alpha_r$}{alpha r}}

\begin{prop}\label{lem.amitsur}
    	Let $p$ be a prime, let $r\geq 1$ be an integer, let $F$ be a field of characteristic exponent $\ell\neq p$, let $\zeta\in F^\times$ be a primitive $p$th root of unity, and suppose that $\Gamma_F$ fits into a short exact sequence
        \begin{equation}\label{eq:amitsur-gamma-f}
        \begin{tikzcd}
        1 \arrow[r] & \Delta \arrow[r] & \Gamma_F \arrow[r] & (\widehat{\Z}^{(\ell')})^{2r} \arrow[r] & 1, 
        \end{tikzcd}
        \end{equation}
        where $\widehat{\Z}^{(\ell')}$ is the prime-to-$\ell$ profinite completion of $\Z$, and where $\Delta$ is a pro-$\ell$-group. Fix scalars $t_1,\dots,t_{2r}\in F^{\times}$ whose images modulo $F^{\times p}$ form a basis of $H^1(F, \mu_p) = F^{\times}/F^{\times p}$ as an $\mathbb{F}_p$-vector space, and let $\alpha_r\coloneqq (t_1,t_2)_p+\dots+(t_{2r-1},t_{2r})_p\in \on{Br}(F)[p]$. Let $F'/F$ be a finite extension of prime-to-$p$ degree, let $L/F'$ be a finite Galois extension, and suppose that $\alpha_r$ splits over $L$. Then \[\on{rank}_p(\on{Gal}(L/F'))\geqslant r.\]
\end{prop}

\begin{rmk}\label{rmk:applies-to-puiseux}
    In the proof of Theorem~\ref{mainthm} in the next section we will be primarily interested in the setting where $k$ is an algebraically closed field of characteristic different from $p$, $t_1, \ldots, t_{2r}$ are algebraically independent variables over $k$, and $F = F_{r} = k(\!(t_1)\!)\cdots(\!(t_{2r})\!)$
    is the field of Laurent power series in $t_1, \ldots, t_{2r}$.
    By \cite[Lemma 5.1]{gille-reichstein}, \Cref{lem.amitsur} applies in this setting;
    cf.~also \cite[Lemma 7.6]{serre2003cohomological}. In this case, 
    \Cref{lem.amitsur} can be readily deduced from~\cite[Corollary 9.5]{tignol2015value}, whose proof uses valuation theory. For the sake of completeness, we include a self-contained cohomological proof of \Cref{lem.amitsur} which was suggested to us by the referee.
\end{rmk}

\begin{proof}[Proof of \Cref{lem.amitsur}]
Throughout the proof, we identify $\mu_p$ with $\F_p$ by means of the primitive
$p$th root of unity $\zeta$. For every finite extension $K/F$, the absolute Galois group of $K$ is also of the form \eqref{eq:amitsur-gamma-f}, with the same $r$. In particular, $H^1(K,\mu_p)$ also has rank $2r$; see~\eqref{h1}. Since $p$ does not divide $[F':F]$, the map $H^1(F,\mu_p)\to H^1(F',\mu_p)$ is an injective linear map of $\F_p$-vector spaces of the same dimension. Hence, this map is bijective; moreover, $t_1,\dots,t_{2r}$ form a basis of $(F')^\times/(F')^{\times p}$. We may thus replace $F$ by $F'$ to reduce to the case when $F'=F$. 

By \eqref{eq:amitsur-gamma-f}, every finite extension of $F$ of prime-to-$\ell$ degree is abelian Galois over $F$, and every finite Galois group over $F$ has normal $\ell$-Sylow subgroup. In particular, letting $P$ be the normal $\ell$-Sylow subgroup of $\on{Gal}(L/F)$ (take $P=\{1\}$ if $\ell=1$), the extension $L^P/F$ is Galois of prime-to-$\ell$ degree, $\on{Gal}(L/F)$ and $\on{Gal}(L^P/F)$ have isomorphic $p$-Sylow subgroups, and the map $\on{Br}(L^P)[p]\to \on{Br}(L)[p]$ is injective, so that $\alpha_r$ splits over $L^P$. We may thus replace $L$ by $L^P$ to reduce to the case when $\on{Gal}(L/F)$ is abelian and of prime-to-$\ell$ order.

Since $\Delta$ is a pro-$\ell$ group and $\ell\neq p$, by the Hochschild--Serre spectral sequence the inflation map $\on{Inf}\colon H^i((\widehat{\Z}^{(\ell')})^{2r},\mu_p)\to H^i(K,\mu_p)$ is an isomorphism for every $i\geq 0$. Combining this with the K\"unneth formula and the fact that $H^*(\widehat{\mathbb Z}^{(\ell')},\F_p)$ is an exterior algebra on a single generator of degree $1$, we deduce that the map $\bigwedge\nolimits^2H^1(F,\mu_p)\to H^2(F,\mu_p)$ given by $a\wedge b\mapsto a\cup b$ is an isomorphism. Since $\Gamma_L$ is also of the form \eqref{eq:amitsur-gamma-f}, we obtain a commutative square
\begin{equation}\label{eq:wedge}
\begin{tikzcd}
    \bigwedge\nolimits^2 H^1(F,\mu_p)\arrow[r,"\sim"]\arrow[d]  & H^2(F,\mu_p) \arrow[d]  \\
    \bigwedge\nolimits^2 H^1(L,\mu_p)\arrow[r,"\sim"] & H^2(L,\mu_p)
\end{tikzcd}
\end{equation}
where the horizontal maps are given by $a\wedge b \mapsto a\cup b$, and where the vertical maps are induced by the inclusion $F\hookrightarrow L$.

Let $e_1,\ldots,e_{2r}$ be the basis of $H^1(F,\mu_p)$ corresponding, under the Kummer isomorphism, to the classes of $t_1,\ldots,t_{2r}$, and consider the symplectic form $\omega\coloneqq e_1\wedge e_2+\cdots+e_{2r-1}\wedge e_{2r}$ on the dual vector space $H^1(F,\mu_p)^\vee$. Let $\varphi\colon H^1(F,\mu_p)\to H^1(L,\mu_p)$ be the restriction homomorphism, and consider its dual $\varphi^\vee\colon H^1(L,\mu_p)^\vee\to H^1(F,\mu_p)^\vee$. Since $(\alpha_r)_{L}=0$ in $H^2(L,\mu_p)$, the pullback $(\varphi^\vee)^*(\omega)$ to $H^1(L,\mu_p)^\vee$ is zero, and so the image of $\varphi^\vee$ is an isotropic subspace of the $2r$-dimensional symplectic vector space $H^1(F,\mu_p)^\vee$. This implies that $\dim_{\mathbb F_p}\on{Im}(\varphi^\vee)\leqslant r$ and hence that $\dim_{\mathbb F_p}(\on{Ker}(\varphi))\geqslant 2r-r=r$. By the inflation-restriction sequence, \[\on{Ker}(\varphi)\cong H^1(\on{Gal}(L/F),\mu_p)\cong \on{Hom}(\on{Gal}(L/F),\F_p).\]
Hence $\dim_{\mathbb F_p}(\on{Hom}(\on{Gal}(L/F),\F_p))\geqslant r$. Since $\on{Gal}(L/F)$ is a finite abelian group, this implies that $\on{Gal}(L/F)$ contains a
subgroup isomorphic to $(\mathbb Z/p\mathbb Z)^r$, as desired.
\end{proof}

	\section{Proof of Theorem \ref{mainthm}}\label{section-5}

	\begin{lemma}\label{prime-to-p-torsor}
		Let $F$ be a field, let $p$ be a prime invertible in $F$, let $A$ be an abelian variety over $F$, and let $T$ be an $A$-torsor over $F$. Then there exists a finite extension $F'/F$ of prime-to-$p$ degree such that the order of $T_{F'}$ in $H^1(F',A)$ is a power of $p$.
	\end{lemma}
	
	\begin{proof}
		Let $\tau$ be the class of $T$ in $H^1(F,A)$. Let $F^{(p)}$ be a $p$-closure of $F$, that is, an algebraic extension of $F$ such that the degree of every finite extension of $F^{(p)}$ is a power of $p$, and such that the degree of every finite extension of $F$ contained in $F^{(p)}$ is not divisible by $p$; see \cite[Proposition 101.16]{ekm}. By a restriction-corestriction argument, $\tau_{F^{(p)}}$ has order $p^e$ in $H^1(F^{(p)},A)$, for some $e\geqslant 0$. Therefore, letting $T'$ be an $A$-torsor over $F$ with class $p^e\cdot \tau$ in $H^1(F,A)$, we have $T'(F^{(p)})\neq \emptyset$ and hence $T'(F')\neq\emptyset$ for some finite subextension $F\subset F'\subset F^{(p)}$. Thus the degree $[F':F]$ is not divisible by $p$ and $p^e\cdot\tau_{F'}=0$ in $H^1(F',A)$, as desired.
	\end{proof}

With Lemma~\ref{prime-to-p-torsor} at hand, we are now ready to finish the proof of \Cref{mainthm}.

		We may assume that $k$ is algebraically closed. Suppose, by contradiction, that there exist an abelian variety $A$ of dimension $g$ over $F_r$ and an $A$-torsor $T$ such that $\alpha_r$ becomes trivial in $\on{Br}(F_r(T))$, or equivalently in $\on{Br}(T)$. By \Cref{prime-to-p-torsor}, there exists a finite extension $F'/F_r$ of prime-to-$p$ degree such that the order of $T_{F'}$ in $H^1(F',A)$ is equal to $p^e$ for some $e\geqslant 0$. By \Cref{split-torsor-abelian-variety}, there exists a Galois field extension $L/F'$ such that $T(L)\neq\emptyset$ and such that $\on{Gal}(L/F')$ is an extension of a subgroup of $\mathrm{SL}_{2g}(\Z/p^e\Z)$ by a subgroup of $(\Z/p^e\Z)^{2g}$. By the subadditivity of the $p$-rank in short exact sequences
		\[\on{rank}_p(\on{Gal}(L/F')) \leqslant \on{rank}_p(\mathrm{SL}_{2g}(\Z/p^e\Z))+\on{rank}_p((\Z/p^e\Z)^{2g}) = \on{rank}_p(\mathrm{SL}_{2g}(\Z/p^e\Z)) + 2g ,\]
		and hence by \Cref{cor-p-rank} we have
		\[
		\on{rank}_p(\on{Gal}(L/F')) \leqslant
		\begin{cases}
			5g^2 +2g-1 & (p>2,\, g\geqslant 2), \\
			9g^2 + 2g-2 & (p = 2,\, g\geqslant 2), \\
			5 & (p>2,\, g=1), \\
			6 & (p=2,\, g=1).
		\end{cases}
		\]
		Combining this with the inequality on $r$ assumed in the statement of \Cref{mainthm}, we conclude that $\on{rank}_p(\on{Gal}(L/F'))<r$. On the other hand, since $T(L)\neq\emptyset$, the pullback map $\on{Br}(L)\to \on{Br}(T_L)$ is injective, and hence $\alpha_r$ splits over $L$. By~\Cref{lem.amitsur} and \Cref{rmk:applies-to-puiseux}, this implies $\on{rank}_p(\on{Gal}(L/F'))\geqslant r$, a contradiction. \qed

\section*{Acknowledgments}
We are grateful to Eoin Mackall, David Saltman, Burt Totaro and the anonymous referee for helpful comments.

\end{document}